\documentclass[preprint,12pt,number]{elsarticle}

\usepackage{amssymb, amsthm, amsmath, amsfonts}
\usepackage{tikz,cancel, enumerate, url, mathrsfs}

\usepackage{amssymb, amsthm, amsmath, amsfonts}
\usepackage{cancel, enumerate}
\usepackage{rotating, environ}
\usepackage{comment}

\NewEnviron{hint}{
	\rotatebox[origin=c]{180}{
		\begin{minipage}[t]{\linewidth}
			\BODY
		\end{minipage}
	}
}

\usepackage[hypertexnames=false]{hyperref}

\usepackage{mathrsfs}     

\newcommand{\al}{\alpha}

\newcommand{\eps}{\varepsilon}

\newcommand{\Om}{\Omega}

\newcommand{\di}{\,\mathrm{d}}

\newcommand{\re}{\mathbb{R}}

\newcommand{\mc}{\mathcal}

\newcommand{\sgn}{\operatorname{sgn}}
\newcommand{\dist}{\operatorname{dist}}
\newcommand{\supp}{\operatorname{supp}}

\newtheorem{theorem}{Theorem}[section]
\newtheorem{lemma}[theorem]{Lemma}
\newtheorem{proposition}[theorem]{Proposition}
\newtheorem{corollary}[theorem]{Corollary}
\newtheorem{thm}{Theorem}[section]
\theoremstyle{remark}
\newtheorem{rem}[thm]{Remark}
\numberwithin{equation}{section}
\theoremstyle{definition}
\newtheorem{defn}{Definition}

\title{A strong maximum principle for nonlinear nonlocal diffusion equations}

\journal{The Journal of Differential Equations}

\begin{document}

\begin{frontmatter}
\author[rvt]{Ravi Shankar\corref{cor1}}
\ead{r.shankar997@gmail.com}
\author[rvt]{Tucker Hartland\corref{cor2}}
\ead{tucker.hartland@gmail.com}

\cortext[cor1]{Principal Corresponding Author}
\cortext[cor2]{Corresponding Author}

\address[rvt]{Department of Mathematics, California State University: Chico, Chico, CA 95929}

\begin{abstract}
This is a study of a class of nonlocal nonlinear diffusion equations.  We present a strong maximum principle for nonlocal time-dependent Dirichlet problems.  Results are for bounded functions of space, rather than (semi)-continuous functions.  Solutions that attain interior global extrema must be identically trivial.  However, depending on the nonlinearity, trivial solutions may not be constant in space; they may have an infinite number of discontinuities, for example.  We give examples of nonconstant trivial solutions for different nonlinearities.  For porous medium-type equations, these functions do not solve the associated classical differential equations, even those in weak form.   We also show that these problems are globally wellposed for Lipschitz, nonnegative diffusion coefficients.
\end{abstract}


\end{frontmatter}

\section{Introduction}\label{sec:intro}
``Nonlocal" interactions take place over large distances, while ``local" interactions are defined by nearby vicinities.  There are many natural phenomena that are described by nonlocal processes. The nonlocal nature of the classical vector potential gives rise to the Aharanov-Bohm effect \cite{StewartA-Bohm}. There are nonlocal diffusive models in epidemological sciences \cite{ahmed2007fractional,reluga2006model,tang2014sei}, and nonlocal dispersive models are applied to the description of material properties \cite{benvenuti2002thermodynamically, eringen1984theory,narendar2016wave, salehipour2015modified}. Further applications of nonlocal diffusion models include: image processing \cite{Gilboa2009,yang2013non,zhou2015evaluation}, visual saliency detection \cite{zhang2015study}, population models \cite{CF,tanzy2015nagumo}, and swarming systems \cite{bernoff2013nonlocal,morale2005interacting,MK}.

Our work concerns the following nonlocal nonlinear diffusion equation for function $u:[0,T]\times\re^N\to\re$:
\begin{align}\label{NDgen}
\begin{split}
u_t&=\int_{\re^N}k\bigl(u(t,x),u(t,y)\bigr)\bigl[u(t,y)-u(t,x)\bigr]J(x-y)\di y.
\end{split}
\end{align}
The kernel $\re^{n}\to\re$ is an integrable, nonnegative function supported in the unit ball.  The diffusion coefficient, or conductivity, $k:\re^2\to\re$ is a locally Lipschitz continuous, nonnegative function.  We seek solutions $u(t,x)$ that are continuous in time but $L^\infty$ in the spatial variables.  Inside a bounded open set $\Om\subset\re^N$, $u$ solves \eqref{NDgen}, but we set $u(t,x)=\psi(x)$ on $\re^N\setminus\Om$.  This is a \emph{Dirichlet} condition in the nonlocal setting, analogous to setting $u(t,x)=\psi(x)$ on $\partial\Omega$ in a classical (local) Dirichlet problem.

Equation \eqref{NDgen} is the nonlocal analogue of a classical nonlinear diffusion equation:
\begin{align}\label{NDclass}
u_t=\nabla\cdot(k(u)\nabla u),
\end{align}
where the diffusion coefficient $k(u)$ accounts for nonlinear material properties, such as flow through a porous medium.  Equation \eqref{NDclass} can be obtained from \eqref{NDgen} by rescaling the kernel $J$ and taking the limit as a nonlocality ``horizon" parameter vanishes, see \cite{Rossi,DGLZ2013}.  It can also be viewed as the first Taylor approximation to \eqref{NDgen} \cite{Coville2005}, provided we take $k(u,u)=k(u)$.

The integro-differential equation \eqref{NDgen} contains, as a special case, the nonlocal \emph{linear} diffusion equation
\begin{align}\label{Dlin}
\begin{split}
u_t&= J \ast u -u\\
&=\int_{\re^N}\bigl[u(t,y)-u(t,x)\bigr]J(x-y)\di y,
\end{split}
\end{align}
where $\int J(z)\di z=1$ represents a probability density, and $u(t,x)$ represents a density for some quantity, such as temperature \cite{Florin} or population density \cite{CF}.  The first term represents the influx of species from a neighborhood surrounding the point $x$, while the second term accounts for population dispersion from $x$.

Equation (\ref{NDgen}) also includes other nonlocal diffusion equations, such as the nonlocal $p-$Laplace equation for $p\ge 2$ covered in \cite[Chapter 6]{Rossi}:
\begin{align}\label{pLap}
u_t=\int_{\re^N}|u(t,y)-u(t,x)|^{p-2}\bigl[u(t,y)-u(t,x)\bigr]J(x-y)\di y,
\end{align}
and the fast diffusion equation of \cite[Chapter 5]{Rossi}:
\[
\partial_t\gamma(u(t,x))=\int_{\re^N}\bigl[u(t,y)-u(t,x)\bigr]J(x-y)\di y.
\]
Assuming $\gamma$ is invertible, letting $v=\gamma(u)$ recovers \eqref{NDgen} by selecting $k(a,b)=(b-a)/(\gamma(b)-\gamma(a))$.  

Another physically interesting equation that can be obtained from \eqref{NDgen} is analogous to the porous medium equation $u_t=\Delta |u|^m/m=\nabla\cdot(|u|^{m-1}\nabla u)$:
\begin{align}
u_t=\int_{\re^N}|u(t,x)+u(t,y)|^{m-1}\bigl[u(t,y)-u(t,x)\bigr]J(x-y)\di y,
\end{align}
where $m>1$; however, to our knowledge, this equation has not been studied in the literature.  The closest that come to it are one using fractional differential operators \cite{Caffarelli}, a nonlinear diffusion PDE with a nonlocal convection term \cite{LiZhang}, and one that is considered in \cite{Bogoya}:
\begin{align}\label{bogoya}
u_t=\int_{\re^N}J\left(\frac{x-y}{u^\alpha(t,y)}\right)u^{1-N\alpha}(t,y)\di y-u(t,x).
\end{align}

Our work mostly concerns the strong maximum principle, a result that provides qualitative information about solutions as well as a tool for developing other \emph{a priori} analytic estimates.  For locally defined diffusion equations such as \eqref{NDclass}, the usual result for classical solutions $u\in C^1[0,T]\cap C^2(\Om)$ is as follows: 

\textit{If $u$ attains a global extremum inside the parabolic cylinder $[0,T]\times\Om$, then $u$ and its boundary values must be identically constant} \cite{Evans}.

An analogous result holds the for the associated time-independent elliptic problem \cite{Pucci,Bobkov}.  There are also similar results for discrete parabolic operators \cite{Mincsovics}, genuinely nonlinear PDEs for semi-continuous functions \cite{Philippin,Gripenberg}, and time-fractional and elliptic fractional differential equations \cite{Ye,Capella,Caffarelli}.  

For integro-differential equations similar to the type in \eqref{NDgen}, there are positivity principles for the linear steady-state problem for $L^2$ functions \cite{Garcia} and strong maximum principles for linear and semilinear traveling wave equations \cite{Coville,Coville2005}, linear and Bellman-type-nonlinear parabolic problems \cite{Paredes}, and degenerate parabolic and elliptic equations for semi-continuous functions involving L\'{e}vy-It\^{o}-type nonlocal operators \cite{Ciomaga,Jakobsen}.  

Our contributions are as follows:
\begin{itemize}
\item{\bf Strong Maximum Principle:} We formulate and prove a strong maximum principle for $L^\infty$ solutions to equations of the type \eqref{NDgen}.  Our results apply to, for example, solutions to a porous medium-type equation with $k(u,v)=|u+v|^{m-1}$, $m\ge 2$, or a $p$-Laplacian equation with $k(u,v)=|u-v|^{p-2}$, $p\ge 3$.  Previous strong maximum principles for parabolic nonlocal diffusion equations have only dealt with \emph{continuous} (or semi-continuous) solutions to equations with \emph{linear} diffusion.

\item{\bf Nonconstant Trivial Solutions:} Our consideration of \emph{nonlinear} diffusion shows that the maximum principle must be altered for certain types of nonlinearities.  Instead of constant solutions, the maximum principle implies the presence of \emph{trivial solutions}, which are, in general, \emph{not constant}; they may have an infinite number of discontinuities.  The precise nature depends on the nonlinearity $k$, and this is shown with examples.

\item{\bf Wellposedness:} We show that the Dirichlet problems considered are globally wellposed using a combination of a local result from the Banach fixed point theorem and the strong maximum principle.
\end{itemize}

In Section \ref{sec:nnde}, we describe the initial boundary value problem and give a local wellposedness proof.  Section \ref{sec:triv} discusses nonconstant trivial solutions and gives examples for different nonlinearities.  Section \ref{sec:smp} gives the proof of the strong maximum principle.  Finally, Section \ref{sec:gwellp} includes the global wellposedness result.

\section{Nonlocal nonlinear diffusion}
\label{sec:nnde}
After precisely outlining the initial-boundary value problem, we give a local wellposedness proof using Banach's fixed point theorem.

\subsection{The initial-boundary value problem}
Let $\Omega$ be a bounded open subset of $\re^N$ and $u:[0,T]\times\re^N\to\re$ be a function of time and space for some $T>0$.  We consider a nonlocal nonlinear diffusion equation:
\begin{align}\label{ND}
\begin{split}
u_t(t,x)&=\int_{\re^N}k\bigl(u(t,x),u(t,y)\bigr)[u(t,y)-u(t,x)]J(x-y)\di y,\\
&(t,x)\in[0,T]\times\Om.
\end{split}
\end{align}
We solve (\ref{ND}) for $u$ subject to a Dirichlet boundary condition
\begin{equation}\label{BC}
u(t,x)=\psi(x),\quad (t,x)\in[0,T]\times(\re^N\setminus\Om)
\end{equation}
and an initial condition
\begin{equation}\label{IC}
u(0,x)=u_0(x), \quad x\in \re^N.
\end{equation}

The kernel $J:\re^N\to\re$ is assumed to be nonnegative and integrable.  We assume $J(z)>0$ if $|z|< 1$, while $J(z)=0$ for $|z|\ge 1$.  
A common example in peridynamics is a \emph{radial} kernel \cite{DGLZ2013}: $J(z)=J(|z|)$, but this is not necessary here. 

The nonlinearity of the problem comes from $k:\re^2\to \re$.  The conductivity $k$ is assumed to be nonnegative.  We take it to be locally Lipschitz continuous in the following sense:
\begin{align}\label{Lip}
|k(u,v)-k(u',v')|\le K(|u|+|u'|,|v|+|v'|)\,\,\,(|u-u'|+|v-v'|),
\end{align}
where $K$ is nonnegative, nondecreasing, and continuous.  The nonnegativity is imposed for global wellposedness.

We will eventually specialize to classes of $k$ defined by
\begin{align}
\label{k3}
\tag{$k3$}
k(u,v)=0\Longrightarrow u=\pm v.
\end{align}
Subclasses include $k$ positive, $k$ positive definite, $k(u,v)=f(u\mp v)$ for $f$ positive definite, etc.  Physically relevant examples include $k(u,v)=|u-v|^{p-2}$ for $p\ge 3$ (nonlocal $p$-Laplacian equation, see \cite{Rossi}), and $k(u,v)=|u+v|^{m-1}$ for $m\ge 2$ (nonlocal porous medium equation).  We are most interested in the latter.


We seek solutions $u(t,x)$ to \eqref{ND} that are in $C([0,T])\cap L^\infty(\re^N)$.  We assume that $u_0(x)$ and $\psi(x)$ are in $L^\infty(\re^N)$ and $L^\infty(\re^N\setminus\Om)$, respectively, and that $u_0(x)=\psi(x)$ when $x$ is in $\re^N\setminus\Om$. We are imposing continuity of $u$ in time, which is the case in many physical applications. The requirement of $u$ being in $L^\infty$ as opposed to, say, $L^p$, is because the strong maximum principle does not have a straightforward interpretation for unbounded functions.  In addition, it allows us to consider more general, integrable kernels.

Using a standard Banach fixed point theorem technique, one can readily prove local wellposedness under the above assumptions.  Global wellposedness is proven in a following section after more preliminary results are established.  Some notations:
\begin{align*}
\|u(t,.)\|&=\text{ess}\sup_{x\in\re^N}|u(t,x)|,\\
\|\|u\|\|&=\max_{t\in[0,T]}\|u(t,.)\|.
\end{align*}
For each $T,\epsilon\ge0$, let $B_{T,\epsilon}$ be the Banach space of $C[0,T]\cap L^\infty(\re^N)$ functions that satisfy $u|_{\re^N\setminus\Om}=\psi$ and $\|\|u\|\|\le \epsilon$.  Let $B_{0,\epsilon}$ contain those $L^\infty(\re^N)$ functions with $\|u\|\le \epsilon$ and $B_{T}:=C[0,T]\cap L^\infty(\re^N)$.

\begin{thm}[Local wellposedness]\label{thm:wellp}
We can find a sufficiently small $T>0$ that depends on $u_0,J,$ and $k$ such that there exists a unique $B_T$ solution $u$ to the initial-boundary value problem \eqref{ND}-\eqref{IC}.
\end{thm}
\begin{proof}
We understand that, for each $u_0$, a $C[0,T]$ solution $u$ to \eqref{ND} and \eqref{IC} satisfies the following integral equation for $t\ge 0$ and $x$ in $\Omega$:
\begin{align}\label{inteqn}
u(t,x)=u_0(x)+\int_0^t\int_{\re^N}k\bigl(u(s,x),u(s,y)\bigr)[u(s,y)-u(s,x)]J(x-y)\di y\di s.
\end{align}
For each $w$ in $B_{0}$, $T\ge 0$, $t\in[0,T]$, and $x$ in $\Omega$, define $\mc A_{w,T}:B_T\to B_T$ as follows:
\begin{align}\label{Adef}
\begin{split}
&\mc A_{w,T}u(t,x):=w(x)\\
&+\int_0^t\int_{\re^N}k\bigl(u(s,x),u(s,y)\bigr)[u(s,y)-u(s,x)]J(x-y)\di y\di s,
\end{split}
\end{align}
with $\mc A_{w,T}u(t,x)=u(t,x)=\psi(x)$ for each $x$ in $\re^N\setminus\Om$.

Our objective is to show that $\mc A_{w,T}$ is a contraction mapping from $B_{T,\epsilon}$ into itself for some $T>0$ and some $0<\epsilon<\infty$.  If such parameters exist, then there exists a unique solution to the equation $u=\mc A_{u_0,T} u$ (i.e. \eqref{inteqn}) by Banach's fixed point theorem.

For brevity, set $k_u=k(u(s,x),u(s,y))$ and $u_x=u(s,x)$.  Then, for each $\epsilon>0$, $u_0,v_0\in B_{0}$, and $u,v\in B_{T}$, we have
\begin{align*}
\|\|\mc A_{u_0,T}u-\mc A_{v_0,T}v\|\|\le \|u_{0}-v_{0}\|+T\|J\|_{L^1}\|\|k_u(u_y-u_x)-k_v(v_y-v_x)\|\|,
\end{align*}
the last supremum taken over $x$ and $y$.  Note the following estimate (cf. \eqref{Lip}):
\begin{align}\label{est}
\begin{split}
&|k_u(u_y-u_x)-k_v(v_y-v_x)|\\
&\le |k_u(u_y-v_y)|+|(k_u-k_v)v_y|+\quad|k_u(v_x-u_x)|+|(k_v-k_u)v_x|\\
&\le K(|u_x|,|u_y|)(|u_x|+|u_y|)(|u_x-v_x|+|u_y-v_y|)\\
&+K(|u_x+v_x|,|u_y+v_y|)(|u_x-v_x|+|u_y-v_y|)(|v_x|+|v_y|)\\
&\le 4(\|\|u\|\|+\|\|v\|\|)K(\|\|u\|\|+\|\|v\|\|,\|\|u\|\|+\|\|v\|\|)\,\|\|u-v\|\|
\end{split}
\end{align}

Thus, for each $\epsilon>0$, $u_0,v_0\in B_{0}$, and $u,v\in B_{T}$:
\begin{align}\label{Aest}
\begin{split}
&\|\|\mc A_{u_0,T}u-\mc A_{v_0,T}v\|\|\le \|u_{0}-v_{0}\|\\
&+4T\|J\|_{L^1}(\|\|u\|\|+\|\|v\|\|)K(\|\|u\|\|+\|\|v\|\|,\|\|u\|\|+\|\|v\|\|)\,\|\|u-v\|\|.
\end{split}
\end{align}

We now fix $u_0$ to be the initial condition of \eqref{IC}.  First, set $v_0=v=0$ in \eqref{Aest}:
\begin{align*}
&\|\|\mc A_{u_0,T}u\|\|\le\|u_{0}\|+4T\|J\|_{L^1}\|\|u\|\|^2\,K(\|\|u\|\|,\|\|u\|\|).
\end{align*}
Set $\epsilon=2\|u_0\|$, and suppose that $u$ is in $B_{T,\epsilon}$.  Then
\begin{align}\label{T1}
\|\|\mc A_{u_0,T}u\|\|&\le \epsilon/2+4T\|J\|_{L^1}\epsilon^2K(\epsilon,\epsilon).
\end{align}
Next, set $v_0=u_0$ in \eqref{Aest}:
\begin{align*}
&\|\|\mc A_{u_0,T}u-\mc A_{u_0,T}v\|\|\le \\
&4T\|J\|_{L^1}(\|\|u\|\|+\|\|v\|\|)K(\|\|u\|\|+\|\|v\|\|,\|\|u\|\|+\|\|v\|\|)\,\|\|u-v\|\|.
\end{align*}
Suppose that $u$ and $v$ are both in $B_{T,\epsilon}$.  We get:
\begin{align}\label{T2}
\begin{split}
&\|\|\mc A_{u_0,T}u-\mc A_{u_0,T}v\|\|\le 8T\|J\|_{L^1}\,\epsilon\,K(2\epsilon,2\epsilon)\,\|\|u-v\|\|.
\end{split}
\end{align}
Estimates \eqref{T1} and \eqref{T2} hold for any $T\ge 0$.  Fix $T$ as follows:
\begin{align}\label{T}
T=\left[\,\,8\,\|J\|_{L^1}\,\epsilon\,\max\left(K(\epsilon,\epsilon)\,\,,\,\,2K(2\epsilon,2\epsilon)\right)\,\,\right]^{-1},
\end{align}
where $\epsilon=2\|u_0\|$.  Then estimates \eqref{T1} and \eqref{T2} respectively become:
\begin{align}\label{cont}
\begin{split}
\|\|\mc A_{u_0,T}u\|\|&\le 2\|u_0\|,\\
\|\|\mc A_{u_0,T}u-\mc A_{u_0,T}v\|\|&\le\|\|u-v\|\|/2.
\end{split}
\end{align}
Thus, for any $u_0\in B_0$, operator $\mc A_{u_0,T}$ maps $B_{T,\epsilon}$ into itself.  It is also a contraction.  The result follows from Banach's fixed point theorem.
\end{proof}

\begin{rem}
This result is not exhaustive by any means.  It is possible to slightly extend this to those $k$ that are (i) Lipschitz in the second argument, (ii) are differentiable away from $u-v=0$, (iii) have $(u-v)k(u,v)$ differentiable at $u-v=0$.  For example, this would include the case of $k(u,v)=|u-v|^{p-2}$ for $p>2$.

Still stronger results may also be obtained.  See e.g. \cite{Rossi} for wellposedness results in the case of $k(u,v)=|u-v|^{p-2}$, where $1\le p$, reaction-diffusion equations \cite{Sun}, and distributional solutions to anomalous diffusion equations \cite{Endal}.  Our analysis, while simple, was able to capture a large variety of nonlinearities $k$.  For example, combined with subsequent analysis, we find that all nonlocal porous medium equations with conductivity $k(u,v)=|u+v|^{m-1}$, $m>2$ are globally wellposed in the strong sense, which, for locally defined differential equations, requires much more effort; see e.g. \cite{Vazquez}.
\end{rem}

We need the following fact to prove the Strong Maximum Principle.
\begin{corollary}[Backwards wellposedness]\label{cor:back}
There exists a unique solution $u=u(t,x)$ on $[-T,0]\times\re^N$ to \eqref{ND} subject to \eqref{IC}-\eqref{BC}.
\end{corollary}
\begin{proof}
Setting $u(t,x)=v(-t,x)$, we find that $v$ satisfies \eqref{ND} but with $k\to -k$.  Since this sign change does not affect the time $T$ in \eqref{T}, we conclude that $v$ exists and is unique on $[0,T]\times\re^N$.
\end{proof}

By replacing $L^\infty(\re^N)$ with $C(\overline\Om)\cap L^\infty(\re^N)$ in the above formulation, we obtain another Banach space, so the wellposedness result also holds for continuous functions.  However, due to the nonlocal nature of \eqref{ND}, the continuity cannot, in general, be extended to the entirety of $\re^N$ (see e.g. \cite{Garcia,Paredes}).  In other words, we set $u=\psi$ on $\re^N\setminus\overline\Om$, rather than on $\re^N\setminus\Om$.  Note the space $C(\overline\Om)$ is the set of \emph{uniformly continuous} functions on $\Om$ extended to $\partial\Om$.  
\begin{corollary}\label{cor:wellpC}
We can find a sufficiently small $T>0$ that depends on $u_0,J,$ and $k$ such that there exists a unique $C([-T,T],C(\overline\Om),L^\infty(\re^N))$ solution $u$ to the initial-boundary value problem \eqref{ND}-\eqref{IC}.
\end{corollary}

One interesting fact is that $C([0,T])$ solutions to \eqref{ND} that solve \eqref{inteqn} automatically gain differentiability in time.
\begin{proposition}[Time differentiability]\label{prop:tdiff}
Let $u$ be a $C([0,T])\cap L^\infty(\re^N)$ solution to \eqref{ND} subject to \eqref{BC}.  Suppose that $k$ is $n$-times continuously differentiable, $n=0,1,2,3,...$.  Then, for any $m=1,2,3,...,n+1$, the time derivative $\partial_t^mu$ is in $C([0,T])\cap L^\infty(\re^N)$; i.e. $u\in C^{n+1}[0,T]\cap L^\infty(\re^N)$.
\end{proposition}
\begin{proof}
Since $\partial_t^nu|_{\re^N\setminus\Om}=0$ for $n=1,2,3,...$ (cf. \eqref{BC}), it is clear that $u$ is in $C^\infty([0,T])\cap L^\infty(\re^N\setminus\Om)$.  First, consider $m=1$.  Since $u$ is in $C([0,T])$, it follows from the right hand side of \eqref{ND} that $\partial_t^1u$ is also in $C([0,T])$.  It is in $L^\infty(\Omega)$, since
\begin{align*}
\|u_t\|_{L^\infty(\Omega)}\le\|k(u(t,x),u(t,y)[u(t,y)-u(t,x)]\|_{L^\infty(\Omega\times\re^N)}\|J\|_{L^1(\re^N)},
\end{align*}
and we assumed $k$ to be continuous.

Assuming that the first $m$ time derivatives of $u$ only have dependence on $J$, continuously differentiable functions of $u$, and integrals thereof, we differentiate \eqref{ND} $m$ times and find that the right hand side depends only on $u,\partial_tu,...,\partial_t^{m-1}u,$ and $\partial_t^mu$, which, in turn, are in $C[0,T]\cap L^\infty(\Om)$ by hypothesis.  The result follows from induction up to $m=n$.
\end{proof}
\begin{corollary}
If $k$ is infinitely differentiable, then $u\in C^\infty([0,T],L^\infty(\re^N))$. In particular, this is the case for the linear equation \eqref{Dlin}, and for \eqref{ND} if $k(u,v)=|u\pm v|^p,p=2,4,6,8,...$.
\end{corollary}

\begin{rem}
This type of result and proof is known for ODEs, such as $u_t=F(u)$, but does not follow so immediately in the classical setting for parabolic PDEs, such as $u_t=\Delta u$, since the right hand side, in that case, would also depend on the spatial derivatives of $u$ (e.g. $u_{tt}=\Delta^2 u$), which would not necessarily be time continuous.
\end{rem}

We need the following related technical lemma for Theorem \ref{thm:smp}.
\begin{lemma}\label{lem:lip}
Let $k$ satisfy \eqref{Lip}.  Then $u_t$ is uniformly Lipschitz continuous in time.
\end{lemma}
\begin{proof}
Choose $t_1$ and $t_2$ in $[0,T]$.  Evaluating \eqref{ND} at these times and subtracting gives:
\begin{align}\label{utlip}
u_t(t_2,x)-u_t(t_1,x)=\int_{\re^N}\bigl[k_2(u_{y}^{2}-u_{x}^{2})-k_1(u_{y}^{1}-u_{x}^{1})\bigr]J(x-y)\di y,
\end{align}
where we have put $k_i=k(u(t_i,x),u(t_i,y))$, and $u_{z}^i=u(t_i,z)$, for brevity.

Proceeding as in \eqref{est}, we have the following estimate for the integrand:
\begin{align*}
\bigl|k_2(u_{y}^2&-u_{x}^2)-k_1(u_{y}^1-u_{x}^1)\bigr|\\
\le &4(\|u^1\|+\|u^2\|)K(\|u^1\|+\|u^2\|,\|u^1\|+\|u^2\|)\,\|u^2-u^1\|.
\end{align*}
By \eqref{cont}, the $L^\infty$ norm of $u$ can be bounded in $[0,T]$, so this becomes:
\begin{align*}
\bigl|k_2(u_{y}^2&-u_{x}^2)-k_1(u_{y}^1-u_{x}^1)\bigr|\le C_1\|u^2-u^1\|.
\end{align*}
By Proposition \ref{prop:tdiff}, $u_t$ is in $C[0,T]\cap L^\infty(\re^N)$.  This means that $u$ is a uniformly Lipschitz continuous function of time:
\begin{align*}
\bigl|k_2(u_{y}^2&-u_{x}^2)-k_1(u_{y}^1-u_{x}^1)\bigr|\le C_2|t_2-t_1|.
\end{align*}
Substitution into \eqref{utlip} gives the desired Lipschitz continuity:
\begin{align*}
\|u_t(t_2,x)-u_t(t_1,x)\|_{L^\infty(\re^N)}\le C_2\|J\|_{L^1(\re^N)}|t_2-t_1|.
\end{align*}
\end{proof}

\section{Trivial solutions}
\label{sec:triv}
When working with continuous solutions to differential equations, the usual notion of a trivial solution is the constant function.  For example, the porous medium equation
\begin{align}\label{pme}
u_t=\Delta(|u|^{m-1}u)=\nabla\cdot(m|u|^{m-1}\nabla u),
\end{align}
admits the trivial solution $u\equiv $ constant.  This is the solution obtained when applying the strong maximum principle to solutions with interior extrema.

For nonlocal equations like \eqref{ND} without spatial differential operators such as $\Delta$ or $\nabla$, this notion of triviality must be modified.  There may be nonconstant trivial solutions.

\begin{defn}[Trivial solutions]\label{def:triv}
We say that $u:\re^N\to\re$ is a trivial solution to \eqref{ND} if it satisfies the following functional equation:
\begin{align}\label{triv}
k(u(x),u(y))[u(y)-u(x)]J(x-y)=0,\quad x\in\Om,\quad y\in\re^N.
\end{align}
\end{defn}
\begin{rem}
Condition \eqref{triv} means that the integrand in \eqref{ND} is identically zero for all time.  Note that the compact support of $J$ means that $u=\psi(y)$ can be left unspecified for those $y$ satisfying $\dist(y,\Om)>1$.  Since they do not influence equation \eqref{ND} in any way, we may define $\psi$ as an equivalence class modulo these values.
\end{rem}
\begin{rem}
Example: porous medium-type equations.  If $k(u,v)=|u+v|^{m-1}$, then a trivial solution can be presented a.e. in the following form:
\begin{align}\label{trivex}
u(x)=U\alpha(x),\quad x\in\re^N,
\end{align}
where $\alpha$ is any function mapping $\re^N$ into $\{-1,1\}$, and $U$ is a constant.  

Interestingly, $\alpha$ may be chosen so that $u$ has infinitely many discontinuities.  For example, if $N=1$, the function
\begin{align}\label{usin}
\al(x)=\begin{cases}
1,&x=0,\pm 1/\pi,\pm 1/(2\pi),...\\
\sgn(\sin(1/x)),& \text{otherwise}
\end{cases}
\end{align}
has an infinite number of discontinuities in any neighborhood of $x=0$.  

In general, functions defined by \eqref{trivex} are not even weak solutions to the classical porous medium equation \eqref{pme}.  A weak solution $u$ to \eqref{pme} solves the following integral equation
\begin{align}\label{weakpme}
\int_{0}^T\int_{\re^N}\left[u\,\partial_t\varphi+|u|^{m-1}u\Delta\varphi\right]\di x\di t=0
\end{align}
for every infinitely differentiable $\varphi$ with compact support in $(0,T)\cap\re^N$.  If, for example, we take $N=1$, $u$ to be that in \eqref{trivex}, and $\alpha(x)=\sgn(x)$, we find that the first term vanishes by the time-constancy of $u$, while, for the second term:
\begin{align*}
\int_0^T\int_{\re}\alpha(x)\,\partial_x^2\varphi(t,x)\di x\di t&=\int_0^T\left[\int_0^\infty\partial_x^2\varphi(t,x)\di x-\int_{-\infty}^0\partial_x^2\varphi(t,x)\di x\right]\di t\\
&=-2\int_0^T\partial_x\varphi(t,0)\di t,
\end{align*}
which, clearly, is not always zero.  For this choice of nonconstant $\alpha$, \eqref{trivex} yields a weak solution only for $U=0$.

This example illustrates that nonlocal nonlinearities do not always give similar results to classical (local) ones.  
It also demonstrates that solutions to \eqref{ND} do not always converge to those of \eqref{pme} as, roughly, the support of $J$ vanishes, which \emph{does} hold in the linear \eqref{Dlin} or $p$-Laplacian \eqref{pLap} cases \cite{Rossi}.  

This example may be contrasted with a less interesting one:
\begin{align*}
\al(x)=\begin{cases}
1,&x\in\mathbb{\re}\setminus\mathbb{Q}\\
-1,& x\in\mathbb{Q},
\end{cases}
\end{align*}
which is equivalent, almost everywhere, to $\alpha(x)\equiv 1$, and, hence, \emph{does} give a weak solution to \eqref{pme}.
\end{rem}
\begin{rem}
To our knowledge, nonconstant trivial solutions such as \eqref{trivex} have not been presented in the literature before.  This is presumably due to the nonlinearities having been considered.  For example, the $p$-Laplacian equation \eqref{pLap} considered in \cite{Rossi} only has one trivial solution: the constant function.  It should also be noted that another porous medium-type equation, given by $k(u,v)=|u|^{m-1}+|v|^{m-1}$, admits only constant trivial solutions.
\end{rem}
\begin{rem}\label{rem:trivGen}
The general solution to \eqref{triv} may be described as follows.  For each $u\in\re$, let $F(u)\subset\re$ be the set of those $v$ such that $k(u,v)=0$.  If $k(u(x),u(y))\bigl[u(y)-u(x)\bigr]=0$ for any $y\neq x$, then either $u(y)\in F(u(x))$ or $u(y)=u(x)$.  Let $U=u(z)$ for some $z\in\re^N$.  Then, for each $y\neq z$, we have $u(y)\in F(U)$ or $u(y)=U$.  Thus, the image of $u$ is $\{U\}\cup F(U)$.

It is clear that $u(x)\equiv U$ is a solution, so assume that there exists some $V\neq U$ such that $u(x)=V$ on a set of positive measure.  Then $V\in F(U)$.  We have $k(U,V)=0$, clearly, but in order for $k(V,U)=0$, we need for $U\in F(V)$.  Provided this is satisfied, a nonconstant solution to \eqref{triv} exists.  We may write it as follows:
\begin{align}
u(x)=U\alpha(x)+\bigl[1-\alpha(x)\bigr]V,
\end{align}
where $\alpha$ is any function from $\re^N$ onto $\{0,1\}$.  Either $u(x)=U$, or $u(x)=V$.

A necessary and sufficient condition for nonconstant solutions to exist is that $U\in F(V)$, provided that $V\in F(U)\setminus\{U\}$.  A necessary condition for this to be true that $U\in F(F(U))$; a common sufficient condition is that $F(V)=F(F(U))$ (provided that $U\in F(F(U))$).

Let us now consider a case for which the essential image of $u$ has more than two elements.  Suppose that $u(x)=W$ on a set of positive measure, where $W\in F(U)\setminus\{U,V\}$.  Clearly, we need for $U\in F(W)$, as with $V$.  However, in order for $k(V,W)=k(W,V)=0$, we need for $V\in F(W)$, and $W\in F(V)$.  Since $V,W\in F(U)$, we see that a necessary condition is that $V,W\in F(F(U))$.  Therefore, the set $F(F(U))\setminus\{U\}$ must contain at least two elements for such solutions $u$ to exist.

For example, if $k(u,v)=|u+v|^{m-1}$, then $F(U)=\{-U\}$.  Let $V=-U$ be in $F(U)$.  Then $F(V)=F(-U)=\{U\}\ni U$, so a non-constant solution exists, namely, \eqref{trivex}.  In fact, since $F(F(U))\setminus\{U\}=\varnothing$, there are no solutions $u$ that attain three or more values, so \eqref{trivex} is the general solution.

A similar example to the previous is if $k(u,v)=|a+uv|^{m}$ for some $a\neq 0$, $m>0$.  In this and the previous case, $F(U)$ defines an involutive function $V:\{V(U)\}=F(U)$, such that $V(V(U))=U$ (since $F(F(U))=\{U\}$), where defined.  Indeed, $F(U)=\{-a/U\}$ for each $U\neq 0$, so $F(F(U))=\{U\}$.  Again, nonconstant solutions exist.  In this case, the nonconstant solution to \eqref{triv} can be presented as:
\begin{align*}
u(x)=U^{-2\alpha(x)+1}(-a)^{\alpha(x)},
\end{align*}
where $\alpha$ is any function from $\re^N$ into $\{0,1\}$, and $U\neq 0$.

A more complicated example is for $k(u,v)=\sin^2\bigl[\pi(u+v)^2\bigr]$.  We have $F(U)=\{-U+n^{1/2}\}_{n=-\infty}^\infty$, where $n^{1/2}:=\sgn(n)\sqrt{|n|}$.  For any $V=-U+m^{1/2}\in F(U)$, we have that $F(V)=\{U-m^{1/2}+n^{1/2}\}_{n=0}^\infty$, so $U\in F(V)$ (take $n=m$).  This means that nonconstant solutions exist for every $U\in\re$.  Two-valued solutions to \eqref{triv} are as follows:
\begin{align*}
u(x)=U\al(x)+\bigl[1-\al(x)\bigr]n^{1/2}/2,
\end{align*}
where $\alpha:\re^N\to\{-1,1\}$, $U$ is an arbitrary constant, and $n$ is an arbitrary integer.

To see if there exists solutions $u$ with images containing more than two elements, we consider $W=-U+\ell^{1/2}\in F(U)$.  For given $U$ and $V$, the set $F(V)$ contains $W$ if $U-m^{1/2}+n^{1/2}=-U+\ell^{1/2}$ has a solution $n$.  In fact, this is the same equation that determines if the set $F(W)$ contains $V$.  This equation has solutions only for countably many values of $U$.  For instance, if $U=\pi$, then there are no solutions.  As an example, suppose that $U=0$, and that $\ell$ and $m$ take values in the signed square integers, $\{p|p|\}_{p=-\infty}^\infty$.  Then this equation admits an integer solution $n$ for every such $\ell,m$.  Therefore, we find the following solution to \eqref{triv}
\begin{align*}
u(x)=\nu(x),
\end{align*}
where $\nu$ is an arbitrary function from $\re^N$ into the integers (i.e $\nu(x)=\ell^{1/2},\ell=0,\pm 1,\pm 4,\pm 9,...$).  
\end{rem}

\begin{rem}\label{rem:semitriv}
In the proof of the Strong Maximum Principle, the following functional equation appears:
\begin{align}\label{semitriv}
k(U,u(y))\bigl[u(y)-U\bigr]J(\chi-y)=0,\quad y\in\re^N,
\end{align}
where $U$ and $\chi\in\overline\Om$ are constants.  This is similar to \eqref{triv}, and, for the proof, we need to know under which conditions solutions to \eqref{semitriv} necessarily satisfy \eqref{triv}.  

For a given $y$, either $u(y)=U$, or $u(y)\in F(U)\setminus\{U\}$.  In order for \\$k(u(x),u(y))\bigl[u(y)-u(x)\bigr]=0$, we need for either $u(y)=u(x),$ or $u(y)\in F(u(x))$.  If $u(x)=U$, then we need for either $u(y)=U$ or $u(y)\in F(U)$.  Both such conditions are satisfied as a result of $u$ solving \eqref{semitriv}.  On the other hand, if $u(x)$ is in $F(U)\setminus\{U\}$, then we need for either $u(y)=u(x)$, or for $u(y)\in F(u(x))\subset F(F(U))$.  If the first condition is not satisfied, then we need for  $u(y)\in F(u(x))\setminus u(x)$, given that $u(y)\in \{U\}\cup F(U)\setminus u(x)$.  

One important case is when $F(U)=\{u(x)\}$, and $F$ acts like an involutive function: $F(F(U))=\{U\}$, given that $u(x)\neq U$, since this implies that $u(y)=U\in F(u(x))=F(F(U))$.  As in Remark \ref{rem:trivGen}, examples for this case are $k(u,v)=|u+v|^{m-1}$ (more generally, those that satisfy \eqref{k3} for the minus sign) as well as $k(u,v)=|1+uv|^m$.
\end{rem}

Thus, for the simple class of coefficients $k$ in \eqref{k3}, we have the following characterization.
\begin{lemma}\label{lem:semitriv}
Suppose that $u$ solves \eqref{semitriv}.  If $k$ satisfies \eqref{k3}, then $u$ is also a trivial solution to \eqref{ND} that solves \eqref{triv}.
\end{lemma}

Finally, trivial solutions satisfy the following obvious property that is important for proving the Strong Maximum Principle.
\begin{lemma}\label{lem:triv}
Let $u$ be a solution to \eqref{ND}.  If, for some $t=t_0$, $u(t_0,x)$ is a trivial solution to \eqref{ND} that solves \eqref{triv}, then $u$ is identically trivial.
\end{lemma}
\begin{proof}
The function $v(t,x)=u(t_0,x)$ is a trivial solution to \eqref{ND} for all times $t$.  By forwards (Theorem \ref{thm:wellp}) and backwards (Corollary \ref{cor:back}) uniqueness of solutions, we can only conclude that $u\equiv v$ is trivial.
\end{proof}

\section{Strong Maximum Principle}\label{sec:smp}
We prove the Strong Maximum Principle for nonlocal equations \eqref{ND} with nonlinearity discussed in Remark \ref{rem:semitriv}.  In particular, this result holds for porous medium-type coefficients $k(u,v)=|u+v|^{m-1}$, $m\ge 2$ and $p$-Laplacian coefficients $k(u,v)=|u-v|^{p-2}$, $p\ge 3$.
 
We first consider continuous functions $u\in C[0,T]\cap C(\overline\Om)\cap L^\infty(\re^N)$.  The proof is more technical in the more general $L^\infty$ case, so this short proof is a useful illustration.
\begin{theorem}[Strong maximum principle]
Let $u$ be a $C([0,T),C(\overline\Om),L^\infty(\re^N))$ solution to (\ref{ND}).  Suppose that $k$ satisfies \eqref{k3}.  If $u$ attains a global extremum in $(0,T]\times\overline\Omega$, then it must be a trivial solution.
\end{theorem}
\begin{proof}
Let $(t_0,x_0)\in(0,T)\times\overline\Omega$ be such that $|u(t_0,x_0)|\ge |u(t,x)|$ for almost all $(t,x)\in[0,T]\times\re^N$.  By virtue of $u(t_0,x_0)$ being an extremum, we see that $u_t(t_0,x_0)=0$.  Substituting this and $(t,x)=(t_0,x_0)$ into \eqref{ND} gives:
\begin{align*}
0=\int_{\re^N}k\bigl(u(t_0,x_0),u(t_0,y)\bigr)\bigl[u(t_0,y)-u(t_0,x_0)\bigr]J(x_0-y)\di y.
\end{align*}
Since $u(t_0,x_0)$ is a global maximum or minimum, the integrand is one-signed, so we deduce that it is identically zero.  By Lemma \ref{lem:semitriv}, $u(t_0,y)$ is a trivial solution, which, from Lemma \ref{lem:triv}, implies that $u$ is identically trivial.
\end{proof}

We now consider the case of $u$ being in $C[0,T]\cap L^\infty(\re^N)$.  For future reference, we set: 
\begin{align*}
u^+(t)&=\|u(t,x)\|_{L^\infty(\re^N)}=\text{ess}\sup_{\re^N}|u(t,x)|,\\
u_-(t)&=\text{ess}\inf_{\re^N}u(t,x),\\
u_+(t)&=\text{ess}\sup_{\re^N}u(t,x),\\
U_-(t)&=\text{ess}\inf_{\Omega}u(t,x),\\
U_+(t)&=\text{ess}\sup_{\Omega}u(t,x),\\
\psi_-&=\text{ess}\inf_{\re^N\setminus\Om}\psi(x),\\
\psi_+&=\text{ess}\sup_{\re^N\setminus\Om}\psi(x).
\end{align*}

We recall the following well known result used in the proof of the Fr\'echet-Kolmogorov Theorem (see e.g. Theorem 2.21 in \cite{Adams}).  
\begin{lemma}\label{lem:adams}
Let $u:\re^N\to\re$ be an integrable function, and let $\tau_au(x)=u(x+a)$ define the translation operator, $a\in\re^N$.  Then $\|\tau_au-u\|_{L^1}\to 0$ as $|a|\to 0$.
\end{lemma}
We also note another well known fact (Corollary 2.32 in \cite{Folland}).
\begin{lemma}\label{lem:folland}
If $f_n\to f$ in $L^1$, then there is a subsequence $\{f_{n_j}\}$ such that $f_{n_j}\to f$ almost everywhere.
\end{lemma}

\begin{thm}[Strong maximum principle]\label{thm:smp}
Let $u$ be a $C([0,T],L^\infty(\re^N))$ solution to \eqref{ND}.  Suppose that $k$ satisfies \eqref{k3}.  If either $U_+(t_0)\ge u_+(t)$ or $U_-(t_0)\le u_-(t)$ for some $t_0\in(0,T]$ and all $t\in[0,t_0]$, then $u$ is a trivial solution.
\end{thm}
\begin{proof}
Consider the first case of $U_+(t_0)\ge u_+(t)$.  We have $U_+(t)=\lim_{n\to\infty}u(t,x_n(t))$ for some sequence $(x_n(t))\subset\Om$.  By the relative compactness of $\Om$, we may redefine $(x_n)$ as one of its convergent subsequences: $x_n(t)\to x(t)\in\overline\Om$. 

Let us evaluate \eqref{ND} at $t=t_0,x=x_n(t_0)$ and make a change of variables, $y=z+x_n(t_0)$, setting $u_n=u(t_0,x_n(t_0))$ for brevity:
\begin{align}\label{eqn}
\begin{split}
u_t(t_0,&x_n(t_0))\\
&=\int_{\re^N}k\bigl(u_n,u(t_0,y)\bigr)\bigl[u(t_0,y)-u_n\bigr]J(x_n(t_0)-y)\di y\\
&=\int_{B_1(0)}k\bigl(u_n,u(t_0,z+x_n(t_0))\bigr)\bigl[u(t_0,z+x_n(t_0))-u_n\bigr]J(-z)\di z,
\end{split}
\end{align}
where the new integration domain over the unit ball comes from $\supp J=B_1(0)$.

We intend to pass to the limit in the integrand.  Since $u$ is in $L^\infty$,  the integrand is uniformly bounded by integrable function $CJ(-z)$, for some constant $C$.  To see that it converges pointwise a.e., we note that the functions
\begin{align*}
u(t_0,z+x_n(t_0))=\tau_{x_n(t_0)-x(t_0)}u(t_0,z+x(t_0)),\quad n=1,2,3,...
\end{align*}
converge to $u(t_0,z+x(t_0))$ in $L^1$ by Lemma \ref{lem:adams} (clearly, they are integrable on $B_1(0)$).  By Lemma \ref{lem:folland}, this means they converge for almost all $z$ (after redefining $(x_n)$ as one of its subsequences).  Finally, since $k$ is a continuous function, the integrand in \eqref{eqn} converges for almost all $z$.

Therefore, the dominated convergence theorem lets us take the limit in \eqref{eqn}:
\begin{align}\label{eval}
\begin{split}
\lim_{n\to\infty}&u_t(t_0,x_n(t_0))\\
&=\int_{B_1(0)}k\bigl(U_+(t_0),u(t_0,z+x(t_0))\bigr)\bigl[u(t_0,z+x(t_0))-U_+(t_0)\bigr]J(-z)\di z\\
&=\int_{\re^N}k\bigl(U_+(t_0),u(t_0,y)\bigr)\bigl[u(t_0,y)-U_+(t_0)\bigr]J(x(t_0)-y)\di y\\
&\le 0,
\end{split}
\end{align}
since $u(t_0,y)\le U_+(t_0)$ for almost all $y$.

Let us consider \eqref{eval} at $t=t_0$.  Either the limit is zero, or it is negative.  If it is zero, then the nonnegativity of the integrand demands that the latter vanish almost everywhere.  That is, $u(t_0,y)$ solves \eqref{semitriv}.  By Lemma \ref{lem:semitriv}, $u(t_0,y)$ also solves \eqref{triv} and is a trivial solution.  By Lemma \ref{lem:triv}, $u$ is identically trivial, in which case the desired result holds.

So, let us suppose, to the contrary, that $u$ is not a trivial solution.  Then the second case holds, and $\lim_{n\to\infty}u_t(t_0,x_n(t_0))< 0$.  This means that
\begin{align}\label{utsmall}
\begin{split}
u_t(t_0,x_n(t_0))<-\epsilon
\end{split}
\end{align}
for some $\epsilon>0$ and all $n$ sufficiently large.  We claim there is some $\delta>0$ such that $u_t(t,x_n(t_0))<-\epsilon/2$ for all such $n$ and each $t\in[t_0-\delta,t_0]$.  If not, then, for each $\delta>0$, there exists $n_\delta\in\mathbb{N}$ and $t_\delta\in[t_0-\delta,t_0]$ such that $u_t(t_\delta,x_{n_\delta}(t_0))\ge-\epsilon/2$.  Subtracting \eqref{utsmall} from this inequality gives
\begin{align*}
u_t(t_\delta,x_{n_\delta}(t_0))-u_t(t_0,x_n(t_0))>\epsilon/2.
\end{align*}
But by Lemma \ref{lem:lip}, the left hand side can be made arbitrarily small by choosing a sufficiently small $\delta$, a contradiction.  

So, the fundamental theorem of calculus gives:
\begin{align*}
u(t,x_n(t_0))-u(t_0,x_n(t_0))&=-\int_{t}^{t_0}u_s(s,x_n(t_0))\di s\\
&>(t_0-t)\epsilon/2
\end{align*}
for any $t\in[t_0-\delta,t_0]$.  Adding and subtracting $U_+(t_0)$ to the lhs and rearranging:
\begin{align*}
u(t,x_n(t_0))-U_+(t_0)>u(t_0,x_n(t_0))-U_+(t_0)+(t_0-t)\epsilon/2.
\end{align*}
We may choose $n$ large enough to make $u(t_0,x_n(t_0))-U_+(t_0)$ arbitrarily small, so we conclude that $u(t,x_n(t_0))>U_+(t_0)>U_+(t)$, a contradiction to $U_+(t_0)$ being a global supremum.

Now, consider the second case of $U_-(t_0)\le u_-(t)$.  We note that $U_-(t)=-(-U)_+(t)$, where $(-U)_+=\sup(-u)$.  This means that $(-U)_+(t_0)\le (-u)_+(t)$ for all $t\in[0,t_0]$.  Now, the function $-u$ also solves \eqref{ND}, but with $\bar k(u,v):=k(-u,-v)$ in place of $k(u,v)$.  It is clear that $\bar k$ also satisfies \eqref{k3}, and the hypothesis of the theorem.  Therefore, $-u$ is a trivial solution, which implies that $u$ is a trivial solution.
\end{proof}

We can use this to deduce a positivity principle similar to those in \cite{Rossi,Garcia}.
\begin{corollary}[Positivity preservation]\label{thm:pos}
Let $u$ be a $C^1([0,T])\cap L^\infty(\re^N)$ solution to \eqref{ND}, for $k$ satisfying \eqref{k3}, and suppose that $u_0(x),\psi(x)\ge 0$ a.e. on $\re^N$ and $\re^N\setminus\Om$.  Then $u(t,x)\ge 0$ a.e. on $[0,T]\times\re^N$.
\end{corollary}
\begin{proof}
If not, then, since $\psi_-\ge 0$ for all $t$, we must have $U_-(t_0)<0$ for some $t_0>0$.  By continuity (Lemma \ref{lem:infcont}), this means there exists a subinterval $[a,b]\subset[0,t_0]$ on which $U_-(t)$ is decreasing and is negative.  Of course, this means that $U_-(b)\le u_-(t)$ for all $t\in[a,b]$.  By the strong maximum principle and solution uniqueness, this implies that $u$ is a trivial solution for all time.  However, since such solutions do not depend on time, this means that $u=u_0\ge 0$, a contradiction to $u$ being negative.
\end{proof}

\section{Global wellposedness}\label{sec:gwellp}
Using the Strong Maximum Principle (Theorem \ref{thm:smp}), we first deduce that the $L^\infty$ norm of $u$ is a nonincreasing function of time.  Combining this with the local wellposedness result (Theorem \ref{thm:wellp}), we show that \eqref{ND} is globally wellposed.

\begin{lemma}\label{lem:infcont}
Let $u$ be a $C([0,T],L^\infty(\re^N))$ solution to \eqref{ND}.  Then $u_-,u_+,U_-(t),$ and $U_+(t)$ are continuous functions.
\end{lemma}
\begin{proof}
Consider the $U_-$ case.  Suppose this is false.  Then $U_-$ is discontinuous at some $t_0$ in $[0,T]$.  

A useful observation is that $u$ is uniformly continuous in time.  That is, there exists a constant $C>0$ such that
\begin{align}\label{uniftime}
|u(s,x)-u(t,x)|\le C|s-t|,\quad x\in\re^N.
\end{align}
Indeed, Proposition \ref{prop:tdiff} shows that $u_t$ is both time-continuous and in $L^\infty(\re^N)$ on the closed interval $[0,T]$, so the mean value theorem gives estimate \eqref{uniftime}.

Now, suppose that $\lim_{t\to t_0^+}U_-(t)$ does not exist (the left hand limit case is analogous).  Then there is some $\epsilon>0$ such that
\begin{align}\label{infcompare}
U_-(t_0)-U_-(t)\ge \epsilon
\end{align}
for infinitely many times $t=t_{1/m},m=1,2,3,...$ inside $[t_0,T]$, with $t_{1/(m+1)}<t_{1/m}$ for each $m$, and $t_{1/m}\to t_0$ as $m\to\infty$.  Without loss of generality, we select this sequence such that $|t_{1/m}-t_0|<\epsilon/(3C)$ (cf. \eqref{uniftime}).

Note that $U_-(t_0)$ must be greater than each $U_-(t_{1/m})$.  We may approximate $U_-(t_0)$ from above using a sequence $(x_{0n})$ such that $u(t_0,x_{0n})\searrow U_-(t_0)$.  If it were true that $U_-(t_{1/m})-U_-(t_0)\ge\epsilon$, then we could simply choose $n$ large enough such that, via \eqref{uniftime}, we would get
\begin{align*}
u(t_{1/m},x_{0n})-U_-(t_0)&\le |u(t_{1/m},x_{0n})-u(t_0,x_{0n})|+|u(t_0,x_{0n})-U_-(t_0)|\\
&<C[\epsilon/(3C)]+\epsilon/3=2\epsilon/3.
\end{align*}
Adding and subtracting $U_-(t_{1/m})$ from the left hand side and bounding from above gives that 
\begin{align*}
u(t_{1/m},x_{0n})-U_{-}(t_{1/m})<-\epsilon/3,
\end{align*}
which would imply that, for any $m$, $U_-(t_{1/m})$ is not the infimum at time $t_{1/m}$.

Now, associate with each time $t_{1/m}$ a sequence $(x_{mn})$ of points in $\Omega$ such that we have $u(t_m,x_{mn})\searrow U_-(t_m)$ as $n\to\infty$.  Then, for any $m$ and $n$, eq. \eqref{infcompare} and the triangle inequality give:
\begin{align*}
\eps\le&|U_-(t_{1/m})-u(t_{1/m},x_{mn})|+|u(t_{1/m},x_{mn})-u(t_0,x_{mn})|\\
&+|u(t_0,x_{mn})-U_-(t_0)|.
\end{align*}
By \eqref{uniftime} and our choice of $t_{1/m}$, the second term is less than $\epsilon/3$.  We choose $n$ so large that the first term is less than $\epsilon/3$.  We find that
\begin{align*}
\epsilon/3< |u(t_0,x_{mn})-U_-(t_0)|.
\end{align*}
Since $U_-(t_0)\le u(t_0,x)$ for almost all $x$ in $\Omega$, this means that
\begin{align*}
u(t_0,x_{mn})>U_{-}(t_0)+\epsilon/3.
\end{align*}
But this combined with \eqref{infcompare} shows that $u(t_0,x_{mn})>U_-(t_{1/m})+4\epsilon/3$ for each $m$, or $|u(t_0,x_{mn})-U_-(t_{1/m})|>4\epsilon/3$.  Since \eqref{uniftime} demands that $|u(t_{1/m},x_{mn})-u(t_0,x_{mn})|<\epsilon/3$ for each $m$, this presents a contradiction.

For the $U_+$ case, we need only note that $U_+(t)=-\text{ess}\inf_{\re^N}(-u(t,x))$.  It follows from the previous analysis that this is a continuous function.

Of course, this means that $u_-(t)$ and $u_+(t)$ are also continuous functions, since $u_-=\min(\psi_-,U_-)$, and $u_+=\max(\psi_+,U_+)$, and the comparison functions $\max/\min$ are, themselves, continuous.
\end{proof}

We now show that the $L^\infty(\re^N)$ norm of bounded solutions to \eqref{ND} decays with time.  Explicit decay estimates were obtained in \cite{IgnatRossi} for the nonlocal Cauchy problem involving linear diffusion and nonlinear convection.
\begin{lemma}[$L^\infty$ decay]\label{lem:Linf}
Let $u$ be a $C([0,T],L^\infty(\re^N))$ solution to \eqref{ND}.  Then $u^+=\|u\|_\infty$ is a nonincreasing function.
\end{lemma}
\begin{proof}
First, we show that $u_+$ is a nonincreasing function.  Assume this is false.  Then, by continuity (cf. Lemma \ref{lem:infcont}), $u_+$ is increasing on some interval $[a,b]\subset[0,T]$.  Since $\psi_+$ is a nonincreasing function, $u_+(t)=\psi_+$ for, at most, one time $t_0\in[a,b]$.  Let us consider a subinterval $[c,d]\subset[a,b]\setminus\{t_0\}$.  Then, for all $t\in[c,d]$, $u_+(t)=U_+(t)$.  This means that $U_+(d)\ge u_+(t)$ for all $t\in[c,d]$, so the strong maximum principle implies that $u$ is a trivial solution.  But such solutions do not depend on time, which contradicts $U_+(d)>U_+(c)$.  

Since $-u$ solves \eqref{ND} with $\bar k(u,v)=k(-u,-v)$ in place of $k(u,v)$, we see that $(-u)_+$ is a nonincreasing function.  Since $-u_-=(-u)_+$, and since $u^+=\max(u_+,-u_-)$, we find that $u^+$ is a nonincreasing function.
\end{proof}

\begin{thm}[Global wellposedness]\label{thm:gwellp}
There exists a unique $C[0,\infty)\cap L^\infty(\re^N)$ solution to \eqref{ND} subject to \eqref{BC} and \eqref{IC}.
\end{thm}
\begin{proof}
Let $u^0$ be the unique solution from Theorem \ref{thm:wellp}.  Let $u^1(t,x)=u^0(t-T,x)$ for $t\ge T$, where $T$ is in \eqref{T}.  Then, by Lemma \ref{lem:Linf}, $\|u^1(0,x)\|_{\infty}\le\|u(0,x)\|$, and $u^1|_{\re^N\setminus\Om}=u|_{\re^N\setminus\Om}=\psi$.  We conclude that a $C[0,T]\cap L^\infty(\re^N)$ solution $u^1$ to \eqref{ND} exists and is unique.  Letting $u^n(t,x)=u^{n-1}(t-T,x)$ for $t\ge nT$ for $n=1,2,3,...$, we see inductively that a unique $C[0,T]\cap L^\infty(\re^N)$ solution $u^n$ to \eqref{ND} exists.  Thus, the function $u$ defined by
\begin{align*}
u(t,x)=u^n(t-nT,x),\quad nT\le t\le (n+1)T,\quad n=0,1,2,3,...
\end{align*}
is a unique solution to \eqref{ND} for all $t\ge 0$.
\end{proof}

\section*{Acknowledgments}
We thank Petronela Radu for useful discussions.  This work was partially funded by the NSF DMS Award 1263132.

\bibliographystyle{abbrv} 

\bibliography{NonlocalREU15}

\begin{thebibliography}{10}

\bibitem{Adams}
R.~Adams.
\newblock {\em Sobolev spaces}.
\newblock Academic Press, 1975.

\bibitem{ahmed2007fractional}
E.~Ahmed and A.~Elgazzar.
\newblock On fractional order differential equations model for nonlocal
  epidemics.
\newblock {\em Physica A: Statistical Mechanics and its Applications},
  379(2):607--614, 2007.

\bibitem{Rossi}
F.~Andreu-Vaillo, J.~M. Maz{\'o}n, J.~D. Rossi, and J.~J. Toledo-Melero.
\newblock {\em Nonlocal diffusion problems}, volume 165 of {\em Mathematical
  Surveys and Monographs}.
\newblock American Mathematical Society, Providence, RI; Real Sociedad
  Matem\'atica Espa\~nola, Madrid, 2010.

\bibitem{benvenuti2002thermodynamically}
E.~Benvenuti, G.~Borino, and A.~Tralli.
\newblock A thermodynamically consistent nonlocal formulation for damaging
  materials.
\newblock {\em European Journal of Mechanics-A/Solids}, 21(4):535--553, 2002.

\bibitem{bernoff2013nonlocal}
A.~J. Bernoff and C.~M. Topaz.
\newblock Nonlocal aggregation models: {A} primer of swarm equilibria.
\newblock {\em SIAM Review}, 55(4):709--747, 2013.

\bibitem{Florin}
F.~Bobaru and M.~Duangpanya.
\newblock The peridynamic formulation for transient heat conduction.
\newblock {\em International Journal of Heat and Mass Transfer},
  53(19):4047--4059, 2010.

\bibitem{Bobkov}
V.~Bobkov and P.~Tak{\'{a}}{\v{c}}.
\newblock A strong maximum principle for parabolic equations with the
  p-{L}aplacian.
\newblock {\em J. Math. Anal. Appl.}, 419(1):218--230, 2014.

\bibitem{Bogoya}
M.~Bogoya.
\newblock A nonlocal nonlinear diffusion equation in higher space dimensions.
\newblock {\em J. Math. Anal. Appl.}, 344(2):601--615, 2008.

\bibitem{Caffarelli}
L.~Caffarelli and J.~V{\'a}zquez.
\newblock Asymptotic behaviour of a porous medium equation with fractional
  diffusion.
\newblock {\em Discrete Contin. Dyn. Syst.}, 29(4):1393--1404, 2011.

\bibitem{Capella}
A.~Capella, J.~D{\'a}vila, L.~Dupaigne, and Y.~Sire.
\newblock Regularity of radial extremal solutions for some non-local semilinear
  equations.
\newblock {\em Comm. Partial Differential Equations}, 36(8):1353--1384, 2011.

\bibitem{CF}
C.~Carrillo and P.~Fife.
\newblock Spatial effects in discrete generation population models.
\newblock {\em Journal of Mathematical Biology}, 50(2):161--188, 2005.

\bibitem{Ciomaga}
A.~Ciomaga.
\newblock On the strong maximum principle for second order nonlinear parabolic
  integro-differential equations.
\newblock {\em Adv. Differential Equations}, 17(7{/}8):635--671, 2012.

\bibitem{Coville}
J.~Coville.
\newblock Remarks on the strong maximum principle for nonlocal operators.
\newblock {\em Electron. J. Differential Equations}, 2008(66):1--10, 2008.

\bibitem{Coville2005}
J.~Coville and L.~Dupaigne.
\newblock On a non-local equation arising in population dynamics.
\newblock {\em Proc. Roy. Soc. Edinburgh Sect. A}, 137(4):727--755, 2007.

\bibitem{DGLZ2013}
Q.~Du, M.~Gunzburger, R.~B. Lehoucq, and K.~Zhou.
\newblock A nonlocal vector calculus, nonlocal volume-constrained problems, and
  nonlocal balance laws.
\newblock {\em Mathematical Models and Methods in Applied Sciences},
  23(03):493--540, 2013.

\bibitem{Endal}
J.~Endal, E.~R. Jakobsen, and F.~del Teso.
\newblock Uniqueness and properties of distributional solutions of nonlocal
  degenerate diffusion equations of porous medium type.
\newblock Preprint, 2015.

\bibitem{eringen1984theory}
A.~Eringen.
\newblock Theory of nonlocal elasticity and some applications.
\newblock Technical report, DTIC Document, 1984.

\bibitem{Evans}
L.~C. Evans.
\newblock {\em Partial Differential Equations}.
\newblock American Math Society, 1998.

\bibitem{Folland}
G.~Folland.
\newblock {\em Real {A}nalysis{:} {M}odern {T}echniques and {T}heir
  {A}pplications}.
\newblock John {W}iley \& {S}ons, {I}nc., 2nd edition, 1999.

\bibitem{Garcia}
J.~Garc{\'i}a-Meli{\'a}n and J.~D. Rossi.
\newblock Maximum and antimaximum principles for some nonlocal diffusion
  operators.
\newblock {\em Nonlin. Anal.}, 71(12):6116--6121, 2009.

\bibitem{Gilboa2009}
G.~Gilboa and S.~Osher.
\newblock Nonlocal operators with applications to image processing.
\newblock {\em Multiscale Modeling \& Simulation}, 7(3):1005--1028, 2009.

\bibitem{Gripenberg}
G.~Gripenberg.
\newblock On the strong maximum principle for degenerate parabolic equations.
\newblock {\em J. Differential Equations}, 242(1):72--85, 2007.

\bibitem{IgnatRossi}
L.~I. Ignat and J.~D. Rossi.
\newblock A nonlocal convection-diffusion equation.
\newblock {\em J. Funct. Anal.}, 251:399--437, 2007.

\bibitem{Jakobsen}
E.~R. Jakobsen and K.~H. Karlsen.
\newblock A {``}maximum principle for semicontinuous functions{"} applicable to
  integro-partial differential equations.
\newblock {\em NoDEA Nonlinear Differential Equations Appl.}, 13(2):137--165,
  2006.

\bibitem{LiZhang}
D.~Li and X.~Zhang.
\newblock On a nonlocal aggregation model with nonlinear diffusion.
\newblock Preprint, 2009.

\bibitem{Mincsovics}
M.~Mincsovics.
\newblock Discrete and continuous maximum principles for parabolic and elliptic
  operators.
\newblock {\em J. Comput. Appl. Math.}, 235(2):470--477, 2010.

\bibitem{MK}
A.~Mogilner and L.~Edelstein-Keshet.
\newblock A non-local model for a swarm.
\newblock {\em Journal of Mathematical Biology}, 38(6):534--570, 1999.

\bibitem{morale2005interacting}
D.~Morale, V.~Capasso, and K.~Oelschl{\"a}ger.
\newblock An interacting particle system modelling aggregation behavior: from
  individuals to populations.
\newblock {\em Journal of mathematical biology}, 50(1):49--66, 2005.

\bibitem{narendar2016wave}
S.~Narendar.
\newblock Wave dispersion in functionally graded magneto-electro-elastic
  nonlocal rod.
\newblock {\em Aerospace Science and Technology}, 2016.

\bibitem{Paredes}
E.~Paredes.
\newblock {\em Some results for nonlocal elliptic and parabolic nonlinear
  equations}.
\newblock PhD thesis, Universidad de Chile, 2014.

\bibitem{Philippin}
G.~Philippin and S.~Vernier{-}Piro.
\newblock Discrete and continuous maximum principles for parabolic and elliptic
  operators.
\newblock {\em Nonlinear Anal.}, 47(1):661--679, 2001.

\bibitem{Pucci}
P.~Pucci and J.~B. Serrin.
\newblock {\em The maximum principle}.
\newblock Springer Science and Business Media, 2007.

\bibitem{reluga2006model}
T.~C. Reluga, J.~Medlock, and A.~P. Galvani.
\newblock A model of spatial epidemic spread when individuals move within
  overlapping home ranges.
\newblock {\em Bulletin of mathematical biology}, 68(2):401--416, 2006.

\bibitem{salehipour2015modified}
H.~Salehipour, A.~Shahidi, and H.~Nahvi.
\newblock Modified nonlocal elasticity theory for functionally graded
  materials.
\newblock {\em International Journal of Engineering Science}, 90:44--57, 2015.

\bibitem{StewartA-Bohm}
A.~Stewart.
\newblock Role of the nonlocality of the vector potential in the
  {A}haronov--{B}ohm effect.
\newblock {\em Canadian Journal of Physics}, 91(5):373--377, 2013.

\bibitem{Sun}
J.~Sun.
\newblock On the existence and uniqueness of positive solutions for a nonlocal
  dispersal population model.
\newblock {\em Electron. J. Differential Equations}, 143:1--9, 2014.

\bibitem{tang2014sei}
Q.~Tang, J.~Ge, and Z.~Lin.
\newblock An {SEI}--{SI} avian--human influenza model with diffusion and
  nonlocal delay.
\newblock {\em Applied Mathematics and Computation}, 247:753--761, 2014.

\bibitem{tanzy2015nagumo}
M.~Tanzy, V.~Volpert, A.~Bayliss, and M.~Nehrkorn.
\newblock A {N}agumo-type model for competing populations with nonlocal
  coupling.
\newblock {\em Mathematical biosciences}, 263:70--82, 2015.

\bibitem{Vazquez}
J.~V{\'a}zquez.
\newblock {\em The porous medium equation{:} mathematical theory}.
\newblock Oxford {U}niversity {P}ress, 2007.

\bibitem{yang2013non}
M.~Yang, J.~Liang, J.~Zhang, H.~Gao, F.~Meng, L.~Xingdong, and S.-J. Song.
\newblock Non-local means theory based perona--malik model for image denosing.
\newblock {\em Neurocomputing}, 120:262--267, 2013.

\bibitem{Ye}
H.~Ye, F.~Liu, V.~Anh, and I.~Turner.
\newblock Maximum principle and numerical method for the multi-term time-space
  {R}iesz-{C}aputo fractional differential equations.
\newblock {\em Appl. Math. Comput.}, 227:531--540, 2014.

\bibitem{zhang2015study}
X.~Zhang, C.~Xu, M.~Li, and R.~K. Teng.
\newblock Study of visual saliency detection via nonlocal anisotropic diffusion
  equation.
\newblock {\em Pattern Recognition}, 48(4):1315--1327, 2015.

\bibitem{zhou2015evaluation}
M.-X. Zhou, X.~Yan, H.-B. Xie, H.~Zheng, D.~Xu, and G.~Yang.
\newblock Evaluation of non-local means based denoising filters for diffusion
  kurtosis imaging using a new phantom.
\newblock {\em PloS one}, 10(2):e0116986, 2015.

\end{thebibliography}






\end{document}